\documentclass[12pt,twoside]{amsart}
\usepackage{amssymb}
\usepackage{amscd}

\newcommand{\Proj}[0]{{\operatorname{Proj}}}

\newcommand{\Bs}[0]{{\operatorname{Bs}}}

\newcommand{\Supp}[0]{{\operatorname{Supp}}}

\newcommand{\mult}[0]{{\operatorname{mult}}}
\newcommand{\Ex}[0]{{\operatorname{Ex}}}
\newcommand{\R}{$\mathbb{R}$}
\newcommand{\Q}{$\mathbb{Q}$}
\newcommand{\Z}{$\mathbb{Z}$}
\newtheorem{thm}{Theorem}[section]

\newtheorem{cor}[thm]{Corollary}
\newtheorem{prop}[thm]{Proposition}

\newtheorem{fact}[thm]{Fact}
\newtheorem{ex}[thm]{Example}

\theoremstyle{definition}
\newtheorem{ques}[thm]{Question}
\newtheorem{dfn}[thm]{Definition}

\newtheorem{rem}[thm]{Remark}  
         
\newtheorem{step}{Step}
\newtheorem{nasi}[thm]{}

\begin{document}

\title[X-METHOD]{The X-method for klt surfaces in positive characteristic} 
\author{Hiromu Tanaka} 
\address{Department of Mathematics, Faculty of 
Science, Kyoto University, 
Kyoto 606-8502 Japan} 
\email{tanakahi@math.kyoto-u.ac.jp}
\thanks{The author is partially supported by 
the Research Fellowships of the Japan Society for the Promotion of Science for Young Scientists (24-1937). }

\begin{abstract}
In this paper, we establish a weak version of 
the Kodaira vanishing theorem for surfaces in positive characteristic. 
As an application, 
we obtain some fundamental theorems in the minimal model theory for klt surfaces. 
\end{abstract}

\maketitle

\tableofcontents

\setcounter{section}{-1}

\section{Introduction}
The X-method is a method 
to prove some fundamental theorems 
in the minimal model theory of characteristic zero.  
For example, in characteristic zero, 
we can show the basepoint free theorem 
by using the X-method, see for example \cite[Chapter~3]{KMM} and \cite[Chapter~3]{KM}. 
The X-method mainly depends on two tools: 
resolution of singularities and 
the Kawamata--Viehweg vanishing theorem, which is a generalization of 
the Kodaira vanishing theorem. 
In positive characteristic, we can use resolution of singularities 
in the case where the dimension of the variety is two or three 
(cf. \cite{CP}). 
But, in positive characteristic, 
there exist counter-examples to 
the Kodaira vanishing theorem 
even in the case where the dimension of the variety is two 
(cf. \cite{Raynaud}). 
Thus, we consider the following question. 
Can we establish a vanishing theorem 
in positive characteristic which 
is sufficient for the X-method? 
If the dimension of the variety is two, then 
we have an affirmative answer. 

\begin{thm}[weak Kodaira vanishing theorem]\label{0kodairatype}
Let $X$ be a smooth projective surface 
over an algebraically closed field 
of positive characteristic. 
Let $A$ be an ample Cartier divisor. 
Let $N$ be a nef Cartier divisor which 
is not numerically trivial. 
If $i>0$ and $m\gg 0$, then 
$$H^i(X, K_X+A+mN)=0.$$
\end{thm}

Moreover, by a standard argument, 
we can generalize this theorem 
to a vanishing theorem of Kawamata--Viehweg type or Nadel type.

\begin{thm}[weak Kawamata--Viehweg vanishing theorem]\label{0kvvtype}
Let $X$ be a smooth projective surface over an algebraically closed field 
of positive characteristic. 
Let $A$ be an ample \R-divisor whose fractional part is simple normal crossing. 
Let $N$ be a nef Cartier divisor which is not numerically trivial. 
If $i>0$ and $m\gg 0$, then 
$$H^i(X, K_X+\ulcorner A\urcorner+mN)=0.$$
\end{thm}

\begin{thm}[weak Nadel vanishing theorem]\label{0nadeltype}
Let $X$ be a normal projective surface 
over an algebraically closed field 
of positive characteristic. 
Let $\Delta$ be an \R-divisor 
such that $K_X+\Delta$ is \R-Cartier. 
Let $N$ be a nef Cartier divisor which 
is not numerically trivial. 
Let $L$ be a Cartier divisor such that 
$L-(K_X+\Delta)$ is nef and big. 
If $i>0$ and $m\gg 0$, then 
$$H^i(X, \mathcal O_X(L+mN)\otimes 
\mathcal J_{\Delta})=0$$
where $\mathcal J_{\Delta}$ is 
the multiplier ideal of the pair $(X, \Delta)$. 
\end{thm}

Using Theorem~\ref{0nadeltype}, 
we obtain the following basepoint free theorem 
(cf. \cite[Theorem~3.3]{KM}). 

\begin{thm}[Basepoint free theorem]
\label{0basepointfree}
Let $X$ be a projective normal surface 
over an algebraically closed field 
of positive characteristic. 
Let $\Delta$ be a \Q-divisor 
such that $\llcorner \Delta\lrcorner=0$ and 
$K_X+\Delta$ is \Q-Cartier. 
Let $D$ be a nef Cartier divisor which 
is not numerically trivial. 
Assume $aD-(K_X+\Delta)$ is nef and big 
for some $a\in \mathbb{Z}_{>0}$. 
Then there exists a positive integer $b_0$ 
such that, if $b\geq b_0$, 
then $|bD|$ is basepoint free. 
\end{thm}

Thus, if we can generalize the above 
vanishing theorems to the case of threefolds, 
then we can prove the above basepoint free theorem for threefolds. 
Unfortunately, however, 
there exists a counter-example 
to the above weak Kodaira vanishing theorem 
in the case where the dimension is three. 
We construct such counter-examples in Section~5. 

By the same argument as the proof of the 
above weak Kodaira vanishing theorem 
(Theorem~\ref{0kodairatype}), 
we can also establish the following vanishing theorem.

\begin{thm}\label{0fiberkvvtype}
Let $\pi:X\to S$ be a morphism 
over an algebraically closed field 
of positive characteristic 
from a smooth projective variety $X$ 
to a projective variety $S$. 
Let $A$ be a $\pi$-ample \R-divisor on $X$ 
whose fractional part is simple normal crossing. 
Set $f_{\max}:=\max_{s\in S} \dim \pi^{-1}(s)$. 
If $i\geq f_{\max}$, then 
$$R^i\pi_*\mathcal O_X(K_X+\ulcorner A\urcorner)=0.$$
\end{thm}

\begin{nasi}[Overview of contents]
In Section~1, we summarize the notations. 
In Section~2, we prove Theorem~\ref{0kodairatype}, 
Theorem~\ref{0kvvtype} and Theorem~\ref{0nadeltype} by using the Frobenius maps 
and the Fujita vanishing theorem. 
In Section~3, we apply these vanishing results 
to the minimal model theory. 
In Section~4, we show 
Theorem~\ref{0fiberkvvtype} and 
other vanishing theorems. 
In Section~5, we construct counter-examples 
to the above vanishing results in the case 
where the dimension is three. 
\end{nasi}

\begin{nasi}[Overview of related literature]
We summarize the literature related to this paper 
with respect to the vanishing theorems and the basepoint free theorem. 

(Vanishing theorem) 
Let us summarize some known results on the Kodaira vanishing theorem and its generalizations. 

In characteristic zero, 
Kodaira establishes the Kodaira vanishing theorem. 
\cite{Kawamata1} and \cite{Viehweg} generalize this result. 
For detailed treatments, 
see \cite[Chapter 1]{KMM}, \cite[Section~2.4, 2.5]{KM} and \cite[Part Three]{Lazarsfeld}. 

In positive characteristic, \cite{Raynaud} shows that 
there exists a counter-example to the Kodaira vanishing theorem. 
\cite{Ekedahl} and \cite{Mukai} deeply investigate 
the counter-examples to the Kodaira vanishing theorem. 
On the other hand, 
there are some positive results on the Kodaira vanishing theorem in positive characteristic. 
For example, 
\cite{Xie} shows that the Kawamata--Viehweg vanishing theorem holds
for rational surfaces. 
In \cite{KK}, Koll\'ar and Kov\'acs prove the relative Kawamata--Viehweg vanishing theorem 
for birational morphisms between surfaces. 
The proof is a calculation of the cohomology for curves. 
We also establish this result in this paper. (See Corollary~\ref{relativekvv}.) 
Our proof depends on the Frobenius maps.

(Basepoint free theorem) 
In characteristic zero, many people contributed to 
the basepoint free theorem 
(cf. \cite{Benveniste} \cite{Kawamata2} \cite{Kawamata3} \cite{Kawamata4} 
\cite{Reid} \cite{Shokurov}). 

In positive characteristic, 
\cite{Keel} shows the basepoint free theorem 
for \Q-factorial threefolds with non-negative Kodaira dimension, 
defined over the algebraic closure of a finite field.
In this paper, we show the basepoint free theorem for klt surfaces. 
To prove this, we establish a weak version of the Kodaira vanishing theorem (Theorem~\ref{0kodairatype}). 

Here, let us compare Theorem~\ref{0basepointfree} 
with the following basepoint free theorem obtained in \cite{Tanaka}. 

\begin{thm}[Theorem~0.3 of \cite{Tanaka}]
\label{0tanakabasepointfree}
Let $X$ be a projective normal 
\Q-factorial surface over 
an algebraically closed field 
of positive characteristic. 
Let $\Delta$ be a \Q-divisor 
such that $\llcorner \Delta\lrcorner=0$. 
Let $D$ be a nef Cartier divisor. 
Assume $aD-(K_X+\Delta)$ is nef and big 
for some $a\in \mathbb{Z}_{>0}$. 
Then $D$ is semi-ample.
\end{thm}
Theorem~\ref{0tanakabasepointfree} 
does not need the assumption that 
$D$ is not numerically trivial. 
On the other hand, 
Theorem~\ref{0basepointfree} 
does not need the \Q-factoriality and 
its claim is stronger than the semi-ampleness. 

The proof of Theorem~\ref{0basepointfree} and 
the one of Theorem~\ref{0tanakabasepointfree} 
are essentially different. 
The proof of Theorem~\ref{0basepointfree} 
depends on the above vanishing theorem (Theorem~\ref{0nadeltype}). 
On the other hand, 
the proof of Theorem~\ref{0tanakabasepointfree} uses 
the minimal model theory for \Q-factorial surfaces. 
In characteristic zero, \cite{Fujino} establishes the minimal model theory for \Q-factorial surfaces. 
In \cite{Tanaka}, the author establishes 
the minimal model theory for \Q-factorial surfaces in positive characteristic. 
The arguments in \cite{Tanaka} heavily depend on \cite[Theorem~0.2]{Keel}, 
which holds only in positive characteristic (cf. \cite[Section~3]{Keel}). 
Keel's proof depends on the Frobenius maps and the theory of the algebraic spaces. 
For alternative proofs of \cite[Theorem~0.2]{Keel}, 
see \cite{CMM} and \cite{FT}. 
\cite{FT} only considers the case of surfaces. 
\end{nasi}

\section{Notations}

We will freely use the notation and terminology of \cite{KM}. 

Our notation will not distinguish 
between invertible sheaves and Cartier divisors.
For example, 
we will write $L+M$ for 
invertible sheaves $L$ and $M$. 

For a coherent sheaf $F$ and a Cartier divisor $L$, 
we define $F(L):=F\otimes \mathcal O_X(L)$. 

Throughout this paper, 
we work over an algebraically closed field $k$ 
of positive characteristic and 
let ${\rm char}\,k=:p>0.$

In this paper, {\em a variety} means 
an integral scheme which is separated and of finite type over $k$. 
{\em A curve} or {\em a surface} means a variety 
whose dimension is one or two, respectively. 

Let $X$ be a projective normal variety and 
let $L$ be a nef \R-Cartier \R-divisor. 
We define {\em the numerical dimension} 
$\nu(X, L)\in\{0,1,\cdots,\dim X\}$ as follows. 
If $L$ is numerically trivial, then 
we set $\nu(X, L)=0$. 
If $L$ is not numerically trivial, then 
we define $\nu(X, L)$ by 
$$\nu(X, L):=\max\{e\in \mathbb Z_{\geq 1}
\,|\, L^{e}\,\,\,
{\rm is\,\,\,not\,\,\,numerically\,\,\,trivial}\}.$$
Note that $L$ is not numerically trivial if 
and only if $\nu(X, L)\geq 1.$

\section{Vanishing theorems for surfaces}

In this section, 
we establish some vanishing theorems for surfaces. 
Proposition~\ref{h1proto} is the key in this section. 
We prove Proposition~\ref{h1proto} 
by using Proposition~\ref{h2vanishing}, 
the Fujita vanishing theorem and the Frobenius maps.

Thus, let us recall the Fujita vanishing theorem 
which is a generalization of the 
Serre vanishing theorem.

\begin{fact}[Fujita vanishing theorem]
\label{Zfujitavanishing}
Let $X$ be a smooth projective variety. 
Let $F$ be a coherent sheaf and 
let $A$ be an ample \Z-divisor. 
Then there exists a positive integer 
$m(F, A)$ such that 
$$H^i(X, F(mA+N))=0$$
for every $i>0$, 
every integer $m\geq m(F, A)$ and 
every nef \Z-divisor $N$. 
\end{fact}

\begin{proof}
See \cite[Theorem~(1)]{Fujita1} or 
\cite[Section~5]{Fujita2}. 
\end{proof}

Since we would like to work over 
\R-divisors, 
let us generalize the Fujita vanishing theorem 
to real coefficients.

\begin{thm}[Fujita vanishing theorem for \R-divisors]
\label{Rfujitavanishing}
Let $X$ be a smooth projective variety. 
Let $F$ be a coherent sheaf and 
let $A$ be an ample \R-divisor. 
Then there exists a positive real number 
$r(F, A)$ such that 
$$H^i(X, F(rA+N))=0$$
for every $i>0$, 
every real number $r\geq r(F, A)$ and 
every nef \R-divisor $N$ 
such that $rA+N$ is a \Z-divisor. 
\end{thm}

\begin{proof}
First, we prove that 
we may assume that $A$ is a \Q-divisor. 
Consider the equation: 
$$A=\frac{1}{2}A+\frac{1}{2}A=A'+A''$$
where $A'$ and $A''$ are ample and 
$A'$ is a \Q-divisor. 
Note that we can find $A'$ and $A''$ by 
changing the coefficients of $(1/2)A$ a little. 
Thus we obtain the desired reduction 
by letting $rA+N=rA'+(N+rA'')$. 

Thus we may assume that $A$ is a \Q-divisor. 
Take a positive integer $m_1$ such that 
$m_1A$ is a \Z-divisor. 
Then we obtain the assertion 
by Fact~\ref{Zfujitavanishing} and the equation 
$rA+N=mm_1A+((r-mm_1)A+N).$
\end{proof}

Let us consider 
the following Serre--Fujita type vanishing theorem 
for surfaces.

\begin{prop}\label{h2vanishing}
Let $X$ be a smooth projective surface and 
let $F$ be a coherent sheaf on $X$. 
Let $N$ be a nef \R-divisor with $\nu(X, N)\geq 1$. 
Then there exists a positive real number $r(F, N)$ 
such that 
$$H^2(X, F(rN+N'))=0$$
for every positive real number $r\geq r(F, N)$ and 
for every nef \R-divisor $N'$ such that $rN+N'$ 
is a \Z-divisor.
\end{prop}

\begin{proof}
Since $X$ is projective, we obtain the following exact sequence:
$$\mathcal O_X^{\oplus s}\to 
F\otimes \mathcal O_X(A) \to 0$$
where $A$ is a sufficiently ample \Z-divisor. 
Tensoring by $\mathcal O_X(-A+rN+N')$, we have 
$$\mathcal O_X(-A+rN+N')^{\oplus s}\to F(rN+N') \to 0.$$
Thus we may assume that $F=:L$ is an invertible sheaf. 
By Serre duality, we have 
$$h^2(X, L+rN+N')=h^0(X, K_X-L-rN-N').$$
Take an ample \Z-divisor $A'$. 
By $\nu(X, N)\geq 1$, we see $N\cdot A'>0.$ 
Then, 
for every sufficiently large number $r$, 
we obtain 
$$(K_X-L-rN-N')\cdot A'<0.$$
This implies $H^0(X, K_X-L-rN-N')=0.$
\end{proof}

Now, we prove the following weak Kodaira vanishing theorem, 
by using the above vanishing result for $H^2$.

\begin{prop}[weak Kodaira vanishing theorem]\label{h1proto}
Let $X$ be a smooth projective surface and 
let $A$ be an ample \R-divisor. 
Let $N$ be a nef \R-divisor with $\nu(X, N)\geq 1$. 
Then there exists a positive real number $r(A, N)$ 
such that 
$$H^1(X, K_X+A+rN+N'))=0$$
for every positive real number $r\geq r(A, N)$ and 
for every nef \R-divisor $N'$ such that 
$A+rN+N'$ is a \Z-divisor.
\end{prop}

\begin{proof}
Consider the following exact sequence 
$$0\to \mathcal B\to F_*\omega_X\to \omega_X \to 0$$
where $F:X\to X$ is the Frobenius map, that is 
the $p$-th power map, and 
$\mathcal B$ is the kernel of $F_*\omega_X\to \omega_X$. 
Considering the composition of the pushforwards 
by $F, F^2,\cdots,F^{e-1}$, 
we obtain 
$$0\to \mathcal B_e\to F^e_*\omega_X\to \omega_X \to 0$$
for some coherent sheaf $\mathcal B_e$. 

Tensoring by $\mathcal O_X(A+rN+N')$, we have 
$$0\to \mathcal B_e(A+rN+N')\to 
F^e_*\omega_X(A+rN+N')\to \omega_X(A+rN+N') \to 0.$$
We can find a large integer $e>0$ such that 
$$H^1(X, F^e_*\omega_X(A+rN+N'))=
H^1(X, \omega_X(p^eA+p^erN+p^eN'))=0.$$
Note that, by the Fujita vanishing theorem, 
we can take $e$ independent of $r$ and $N'$. 
By Proposition~\ref{h2vanishing}, we have 
$$H^2(X, \mathcal B_e(A+rN+N'))=0$$
for every large $r$. 
These imply
$$H^1(X, \omega_X(A+rN+N'))=0.$$
\end{proof}

In order to generalize the above weak Kodaira 
vanishing theorem 
to a vanishing theorem of Kawamata--Viehweg type, 
we recall the following covering lemma.

\begin{prop}\label{cycliccover}
Let $X$ be an $n$-dimensional smooth variety. 
Let $D$ be a \Q-divisor such that the support of 
the fractional part $\{ D\}$ is 
simple normal crossing. 
Moreover suppose that, 
for the prime decomposition 
$\{ D\}=\sum_{i\in I}
\frac{b^{(i)}}{a^{(i)}}D^{(i)}$, no integers $a^{(i)}$ are divisible by $p$. 
Then there exists a finite surjective morphism $\gamma:Y \to X$ 
from a smooth variety $Y$ with the following properties.
\begin{enumerate}
\item{The field extension $K(Y)/K(X)$ is a Galois extension.}
\item{$\gamma^*D$ is a \Z-divisor.}
\item{$\mathcal O_X(K_X+\ulcorner D\urcorner)\simeq 
(\gamma_*\mathcal O_Y(K_Y+\gamma^*D))^G$, where 
$G$ is the Galois group 
of $K(Y)/K(X)$.}
\item{If $D'$ is a \Q-divisor 
such that $\{D'\}=\{D\}$, then 
$\gamma^*D'$ is a \Z-divisor and 
$\mathcal O_X(K_X+\ulcorner D'\urcorner)\simeq 
(\gamma_*\mathcal O_Y(K_Y+\gamma^*D'))^G$.}
\end{enumerate}
\end{prop}

\begin{proof}
See \cite[Theorem~1-1-1]{KMM}. 
\end{proof}

Now, we can generalize 
the above weak Kodaira vanishing (Proposition~\ref{h1proto}) 
to the following weak Kawamata--Viehweg vanishing.

\begin{thm}[weak Kawamata--Viehweg vanishing theorem]\label{kvvtype}
Let $X$ be a smooth projective surface. 
Let $B$ be a nef and big \R-divisor 
whose fractional part is simple normal crossing. 
Let $N$ be a nef \R-divisor with $\nu(X, N)\geq 1$. 
Then there exists a positive real number $r(B, N)$ 
such that 
$$H^i(X, K_X+\ulcorner B\urcorner+rN+N')=0$$
for every $i>0$, 
every positive real number $r\geq r(B, N)$ 
and every nef \R-divisor $N'$ 
such that $rN+N'$ is a \Z-divisor.
\end{thm}

\begin{proof}
If $i=2$, then 
the assertion follows from 
Proposition~\ref{h2vanishing}. 
Thus we assume $i=1$. 

\begin{step}
In this step, 
we assume that $B=:A$ is ample and 
we prove the assertion. 

Since $A$ is ample, 
we may assume that 
$A$ is an ample \Q-divisor and 
that no denominators of the coefficients of 
its fractional part 
are divisible by $p$. 
Note that the fractional part of 
$A+rN+N'$ is equal to 
the fractional part of $A$ 
for an arbitrary real number $r$ and 
for a nef \R-divisor $N'$ such that 
$rN+N'$ is a \Z-divisor. 
Thus we can apply Proposition~\ref{cycliccover} 
for $D:=A+rN+N'$ and 
we obtain a finite cover $\gamma:Y\to X$ 
with the properties in the proposition. 
Note that the map $\gamma$ is 
independent of $r$ and $N'$. 
Therefore we have 
\begin{eqnarray*}
&&H^1(X, K_X+\ulcorner A\urcorner+rN+N')\\
&=&H^1(X, K_X+\ulcorner (A+rN+N')\urcorner)\\
&=&H^1(X, \gamma_*\mathcal O_Y(K_Y+\gamma^*(A+rN+N')))^G\\
&=&H^1(Y, K_Y+\gamma^*A+r\gamma^*N+\gamma^*N')^G\\
&=&0.
\end{eqnarray*}
The last equality follows 
from Proposition~\ref{h1proto} when $r\gg 0$.
\end{step}

\begin{step}
In this step, we prove the assertion. 

Let $f:Y\to X$ be a birational morphism 
from a smooth projective surface 
with the following properties: 
there exists an effective \Z-divisor $E$ 
such that $f^*B-\epsilon E$ 
is ample for $0<\epsilon\ll 1$
and the fractional part $\{f^*B-\epsilon E\}$ 
is simple normal crossing. 
Since $f$ has a decomposition into blow-ups of points, 
we consider the blow-up $g:Z\to X$ of one point $P$. 
Let $C$ be the exceptional curve. 
Set $\Delta_X:=\ulcorner B\urcorner-B$ and 
$M:=\Delta_X+B+rN+N'$. 
We will prove that 
$$H^1(X, K_X+\Delta_X+B+rN+N')=0.$$
Consider the exact sequences induced 
from the corresponding Leray spectral sequences: 
$$0\to H^1(X, K_X+M)\to H^1(Z, K_Z-C+g^*M)$$
$$0\to H^1(X, K_X+M)\to H^1(Z, K_Z+g^*M).$$
Note that the second exact sequence 
is obtained by Serre duality. 
If $\mult_P\Delta_X\geq 1$, then 
we set $\Delta_Z:=g^*(\Delta_X)-C$ and 
we can reduce the problem on $X$ to the problem on $Z$ 
by the first exact sequence. 
If $\mult_P\Delta_X< 1$, then 
we set $\Delta_Z:=g^*(\Delta_X)$ 
and we can also reduce 
the problem on $X$ to the problem on $Z$ by the second exact sequence. 
Thus it is sufficient to prove that 
$$H^1(Y, K_Y+\Delta_Y+f^*(B+rN+N'))=0.$$
Note that $\llcorner \Delta_Y\lrcorner=0.$ 
We see 
\begin{eqnarray*}
&&H^1(Y, K_Y+\Delta_Y+f^*(B+rN+N'))\\
&=&H^1(Y, K_Y+\ulcorner f^*B\urcorner+f^*(rN+N'))\\
&=&H^1(Y, K_Y+\ulcorner f^*B-\epsilon E\urcorner
+f^*(rN+N'))\\
&=&0.
\end{eqnarray*}
The first equality follows from $\llcorner \Delta_Y\lrcorner=0$. 
The third equality follows from Step~1 
when $r\gg 0$. 
\end{step}
\end{proof}

By this theorem, 
we obtain the relative Kawamata--Viehweg 
vanishing theorem 
for non-trivial morphisms. 

\begin{cor}\label{relativekvv}
Let $\pi:X\to S$ be a proper morphism 
from a smooth surface $X$ to a variety $S$. 
Let $B$ be a $\pi$-nef and $\pi$-big \R-divisor whose 
fractional part is simple normal crossing. 
Assume $\dim \pi(X)\geq 1.$ 
Then 
$$R^i\pi_*\mathcal O_X(K_X+\ulcorner B\urcorner)=0$$
for every $i>0$. 
\end{cor}

\begin{proof}
By the same argument as Step~2 of Theorem~\ref{kvvtype}, 
we may assume that 
$B=:A$ is $\pi$-ample. 
We may assume that $S$ is affine. 
Moreover, by taking suitable compactifications of $S$ 
and $X\to S$, 
we may assume that $X$ and $S$ are projective. 
(See, for example, the proof of \cite[Theorem~1-2-3]{KMM}.)
Let $A_S$ be an ample invertible sheaf on $S$ and 
set $N:=\pi^*A_S.$ 
Then $\nu(X, N)\geq 1$. 
Therefore the assertion follows from 
Theorem~\ref{kvvtype} and the following Leray spectral sequence 
\begin{eqnarray*}
E_2^{i, j}:=H^i(S, R^j\pi_*\mathcal O_X(K_X+\ulcorner B\urcorner)\otimes A_S^{\otimes m})\\
\Rightarrow 
H^{i+j}(X, K_X+\ulcorner B\urcorner+m\pi^*A_S)=:E^{i+j}.
\end{eqnarray*}
\end{proof}

In order to generalize the above weak Kawamata--Viehweg vanishing theorem 
to a vanishing theorem of Nadel type, we recall the definition of the multiplier ideals.

\begin{dfn}\label{defmultiplier}
Let $X$ be a normal surface and 
let $\Delta$ be an \R-divisor on $X$ such that 
$K_X+\Delta$ is \R-Cartier. 
Let $\mu:X'\to X$ be 
a log resolution of $(X, \Delta).$ 
We define {\em a multiplier ideal sheaf} 
$\mathcal J_{\Delta}$ 
by $$\mathcal J_{\Delta}:=
\mu_*\mathcal O_{X'}(K_{X'}-
\llcorner\mu^*(K_X+\Delta) \lrcorner).$$
Note that, in the case of surfaces, 
we can use the resolution of singularities in positive characteristic  (cf. \cite{Lipman2}). 
Thus, we can establish some fundamental properties (cf. \cite[Chapter~9]{Lazarsfeld}). 
For example, we see that 
$\mathcal J_{\Delta}$ is independent of log resolutions and that 
if $\Delta\geq 0$, then 
$\mathcal J_{\Delta}\subset \mathcal O_X$. 
\end{dfn}

Now, we prove the weak Nadel vanishing theorem, which is the main theorem in this section.

\begin{thm}[weak Nadel vanishing theorem]\label{nadeltype}
Let $X$ be a projective normal surface and 
let $\Delta$ be an \R-divisor such that 
$K_X+\Delta$ is \R-Cartier. 
Let $N$ be a nef \R-Cartier \R-divisor with 
$\nu(X, N)\geq 1$. 
Let $L$ be a Cartier divisor such that 
$L-(K_X+\Delta)$ is nef and big. 
Then there exists a positive 
real number $r(\Delta, L, N)$ 
such that 
$$H^i(X, \mathcal O_X(L+rN+N')
\otimes \mathcal J_{\Delta})=0$$
for every $i>0$, every positive real number 
$r\geq r(\Delta, L, N)$ and 
every nef \R-Cartier \R-divisor $N'$ such that 
$rN+N'$ is a Cartier divisor.
\end{thm}

\begin{proof}
Let $\mu:X'\to X$ be a log resolution 
of $(X, \Delta).$ 
Set 
$$M:=\mu^*(L+rN+N')+K_{X'}-
\llcorner\mu^*(K_X+\Delta) \lrcorner.$$
Consider the following Leray spectral sequence: 
$$E^{i,j}_2:=H^i(X, R^j\mu_*\mathcal O_{X'}(M))
\Rightarrow H^{i+j}(X', \mathcal O_{X'}(M))=:E^{i+j}.$$
The assertion is equivalent to $E^{i,0}=0.$ 
We see 
$$M=K_{X'}+\ulcorner\mu^*(L-(K_X+\Delta)) 
\urcorner+r\mu^*N+\mu^*N'.$$ 
Thus, by Theorem~\ref{relativekvv}, 
we have $E^{i,j}_2=0$ for $j>0$. 
This means $E^{i,0}_2=E^i$. 
Moreover, by Theorem~\ref{kvvtype}, 
we see that $E^i=0$ for $r\gg 0$. 
\end{proof}

\begin{thm}\label{relativenadel}
Let $\pi:X\to S$ be a proper morphism 
from a normal surface $X$ to a variety $S$. 
Let $\Delta$ be an \R-divisor such that 
$K_X+\Delta$ is \R-Cartier. 
Let $L$ be a Cartier divisor 
such that $L-(K_X+\Delta)$ is 
$\pi$-nef and $\pi$-big. 
Assume $\dim \pi(X)\geq 1.$ 
Then 
$$R^i\pi_*(\mathcal O_X(L)\otimes \mathcal J_{\Delta})=0$$
for every $i>0$. 
\end{thm}

\begin{proof}
We may assume $i=1$. 
Let $\mu:X'\to X$ be a log resolution 
of $(X, \Delta).$ 
We have 
$$0\to R^1\pi_*(\mathcal O_X(L)
\otimes \mathcal J_{\Delta})
\to R^1(\pi\circ\mu)_*
(\mathcal O_{X'}(K_{X'}+\ulcorner\mu^*(L-(K_X+\Delta)) 
\urcorner))$$
by the exact sequence induced from 
the corresponding Grothendieck--Leray spectral sequence. 
The latter term vanishes by 
Corollary~\ref{relativekvv}. 
\end{proof}

The following two results are vanishing theorems of Kawamata--Viehweg type 
for klt surfaces.

\begin{thm}\label{kltkvv}
Let $(X, \Delta)$ be a projective klt surface 
where $\Delta$ is an effective \R-divisor. 
Let $N$ be a nef \R-Cartier \R-divisor with 
$\nu(X, N)\geq 1$. 
Let $D$ be a \Q-Cartier \Z-divisor such that 
$D-(K_X+\Delta)$ is nef and big. 
Then there exists a positive 
real number $r(\Delta, D, N)$ 
such that 
$$H^i(X, \mathcal O_X(D+rN+N'))=0$$
for every $i>0$, every positive real number 
$r\geq r(\Delta, D, N)$ and 
every nef \R-Cartier \R-divisor $N'$ such that 
$rN+N'$ is a Cartier divisor.
\end{thm}

\begin{proof}
Let $\mu:X'\to X$ be a log resolution 
of $(X, \Delta).$ 
Set 
$$M:=\ulcorner\mu^*(D+rN+N')+K_{X'}-\mu^*(K_X+\Delta) \urcorner.$$
Consider the following Leray spectral sequence: 
$$E^{i,j}_2:=H^i(X, R^j\mu_*\mathcal O_{X'}(M))
\Rightarrow H^{i+j}(X', \mathcal O_{X'}(M))=:E^{i+j}.$$
The assertion is equivalent to $E^{i,0}_2=0$ because 
\begin{eqnarray*}
&&\mu_*\mathcal O_{X'}(M)\\
&=&\mu_*\mathcal O_{X'}(\ulcorner\mu^*(D+rN+N')+K_{X'}-\mu^*(K_X+\Delta) \urcorner)\\
&=&\mu_*\mathcal O_{X'}(\llcorner\mu^*(D+rN+N')\lrcorner+
({\rm effective\,\,exceptional\,\,\mathbb Z-divisor}))\\
&\simeq&\mathcal O_{X}(D+rN+N').\\
\end{eqnarray*}
The above second equality holds 
because $(X, \Delta)$ is klt and $D$ is a \Z-divisor. 
We see 
$$M=K_{X'}+\ulcorner\mu^*(D-(K_X+\Delta)) 
\urcorner+r\mu^*N+\mu^*N'.$$ 
Thus, by Theorem~\ref{relativekvv}, 
we have $E^{i,j}_2=0$ for $j>0$. 
This means $E^{i,0}_2=E^i$. 
Moreover, by Theorem~\ref{kvvtype}, 
we see that $E^i=0$ for $r\gg 0$. 
\end{proof}

\begin{thm}\label{relativekltkvv}
Let $\pi:X\to S$ be a proper morphism 
from a normal surface $X$ to a variety $S$. 
Assume that 
$(X, \Delta)$ is a klt surface 
where $\Delta$ is an effective \R-divisor. 
Let $D$ be a \Q-Cartier \Z-divisor 
such that $D-(K_X+\Delta)$ is 
$\pi$-nef and $\pi$-big. 
Assume $\dim \pi(X)\geq 1.$ 
Then 
$$R^i\pi_*(\mathcal O_X(D))=0$$
for every $i>0$. 
\end{thm}

\begin{proof}
We may assume $i=1$. 
Let $\mu:X'\to X$ be a log resolution 
of $(X, \Delta).$ 
We have 
$$0\to R^1\pi_*(\mathcal O_X(D))
\to R^1(\pi\circ\mu)_*
(\mathcal O_{X'}(K_{X'}+\ulcorner\mu^*(D-(K_X+\Delta)) 
\urcorner))$$
by the exact sequence induced from the corresponding 
Grothendieck--Leray spectral sequence and the proof of Theorem~\ref{kltkvv}. 
The latter term vanishes by 
Corollary~\ref{relativekvv}. 
\end{proof}

\section{X-method for surfaces}

In this section, 
we apply the vanishing theorems 
which are established in Section~2 
to the minimal model theory. 
First, we see the 
non-vanishing theorem.

\begin{thm}[Non-vanishing theorem]\label{nonvanishing}
Let $(X, -G)$ be a projective klt surface 
where $G$ is a \Q-divisor. 
Note that $-G$ may not be effective. 
Let $D$ be a nef Cartier divisor $D$ 
such that $\nu(X, D)\geq 1$ and 
$aD-(K_X-G)$ is 
nef and big for some $a\in \mathbb Z_{>0}$. 

Then there exists a positive integer $m_0$ 
such that 
$$H^0(X, mD+\ulcorner G\urcorner)\neq 0.$$
for $m\geq m_0.$
\end{thm}

\begin{proof}
Since the proof is almost identical to that of \cite[Theorem~3.4]{KM}, 
we will only discuss the necessary changes to their argument. 
The numbers of \lq\lq Step"  
are the same as \cite[Theorem~3.4]{KM}. 

The argument of Step~0 works without any changes. 
Because we assume $\nu(X, D)\geq 1$, 
there is nothing to prove in Step~1. 
The arguments of Step~2, Step~3, Step~4 and Step~5 
work without any changes. 

In Step~6, we modify the argument a little. 
It is sufficient to prove 
$$H^1(Y, K_Y+\ulcorner N(b, c)\urcorner)=0$$
where $Y$ and $N(b, c)$ are the notations in \cite[Theorem~3.4]{KM}. 
Note that, by Step~4 and Step~5, we may assume that $b$ is sufficiently large. 
Then, by Theorem~\ref{kvvtype} and the definition of $N(b, c)$, 
we obtain the above vanishing result for every $b\gg 0$.
\end{proof}

Second, 
we prove the following basepoint free theorem.

\begin{thm}[Basepoint free theorem] 
\label{basepointfree}
Let $X$ be a projective normal surface and 
let $\Delta$ be a \Q-divisor 
such that $\llcorner \Delta\lrcorner=0$ and 
$K_X+\Delta$ is \Q-Cartier. 
Let $D$ be a nef Cartier divisor 
such that $\nu(X, D)\geq 1$. 
Assume $aD-(K_X+\Delta)$ is nef and big 
for some $a\in \mathbb{Z}_{>0}$. 
Then there exists a positive integer $b_0$ 
such that if $b\geq b_0$, 
then $|bD|$ is basepoint free. 
\end{thm}

\begin{proof}
If the pair $(X, \Delta)$ is klt, then 
the proof of \cite[Theorem~3.3]{KM} works 
by the same modification as 
Theorem~\ref{nonvanishing}. 
Thus we assume that 
the pair $(X, \Delta)$ is not klt. 
Consider the following exact sequence:
$$0\to \mathcal J_{\Delta}(bD)
\to \mathcal O_X(bD)
\to \mathcal O_{M_{\Delta}}(bD)\to 0
$$
where $M_{\Delta}$ is the closed subscheme 
corresponding to $\mathcal J_{\Delta}$. 
Note that $\Supp M_{\Delta}$ 
consists of the non-klt points. 
In particular, the dimension of $M_{\Delta}$ 
is zero. 
We can apply Theorem~\ref{nadeltype} 
for $L:=aD$ and $N:=D$. 
Then we see that 
there exists a positive integer $b_1$ 
such that if $b\geq b_1$, then 
$H^0(X, bD)\neq 0$ and 
the base locus of $|bD|$ contains 
no non-klt points. 

The following argument 
is a slight modification of \cite[Theorem~3.3]{KM}. 
Fix an arbitrary prime number $q$. 
Let $s$ be a positive integer such that 
$$\Bs|q^sD|=\bigcap_{l\geq 1} \Bs|q^lD|.$$
Note that, since $X$ is a noetherian scheme,
we can find such an integer $s>0$. 
It is sufficient to prove that 
$\Bs|q^sD|=\emptyset.$ 
Suppose the contrary and we derive a contradiction. 
Set $m:=q^s.$ 
By the above argument, 
we see that $\Bs|mD|$ contains no non-klt points. 

Let $f:Y\to X$ be a log resolution 
of $(X, \Delta)$ such that 
\begin{enumerate}
\item{$K_Y=f^*(K_X+\Delta)+\sum a_jF_j$.}
\item{$f^*(aD-(K_X+\Delta))-\sum p_jF_j$ is ample, 
where $0<p_j\ll 1$.}
\item{$f^*|mD|=|L|+\sum r_jF_j$, where 
$|L|$ is basepoint free and 
$\bigcup F_j$ is the fixed locus of $f^*|mD|$.}
\end{enumerate}
Since $\Bs|mD|$ contains no non-klt points, for every $j$, 
the inequality $r_j>0$ implies $a_j>-1.$ 
We define the \Q-divisor $N(b, c)$ by 
\begin{eqnarray*}
N(b, c)&:=&bf^*D-K_Y+\sum(-cr_j+a_j-p_j)F_j\\
&=&(b-cm-a)f^*D\\
&+&c(mf^*D-\sum r_jF_j)\\
&+&f^*(aD-(K_X+\Delta))-\sum p_jF_j. 
\end{eqnarray*}
If $b\geq cm+a$, then 
$N(b, c)$ is ample. 
Thus, for $b\gg 0$, we have 
$$H^1(Y, K_Y+\ulcorner N(b, c)\urcorner)=0$$
by Theorem~\ref{kvvtype}. 

By a small perturbation of $p_j$, 
we can find $c>0$ and a prime divisor $F$ 
in the fixed locus of $f^*|mD|$, which 
satisfy the following property: 
$\sum_{a_j> -1}(-cr_j+a_j-p_j)F_j=:A-F$
where $\ulcorner A\urcorner$ is effective and 
$F$ is not a prime component of $A$. 
Note that we can find such a number $c$ 
because the inequality $r_j>0$ implies $a_j>-1.$ 
Set 
$$G:=-\ulcorner \sum_{a_j\leq -1}(-cr_j+a_j-p_j)F_j\urcorner.$$ 
Note that $G$ is an effective $f$-exceptional \Z-divisor 
and $f(G)$ consists of non-klt points. 
This means $\Supp G \cap \Supp F=\emptyset$ 
because $\Bs|mD|$ contains no non-klt points. 
Then we have 
$$K_Y+\ulcorner N(b, c)\urcorner=bf^*D+
\ulcorner A\urcorner-(F+G).$$ 
Consider the exact sequence: 
\begin{eqnarray*}
0&\to& \mathcal O_X(K_Y+\ulcorner N(b, c)\urcorner)\\
&\to&\mathcal O_X(bf^*D+
\ulcorner A\urcorner)\\ 
&\to&\mathcal O_{F+G}(bf^*D+
\ulcorner A\urcorner)\to 0
\end{eqnarray*}
If $b\gg 0$, then 
$H^1$ of the first term 
$\mathcal O_X(K_Y+\ulcorner N(b, c)\urcorner)$ 
vanishes. 
Let us consider the third term 
$\mathcal O_{F+G}(bf^*D+
\ulcorner A\urcorner)$. 
Since $F$ is disjoint from $G$, 
we have 
$$\mathcal O_{F+G}(bf^*D+
\ulcorner A\urcorner)=
\mathcal O_{F}(bf^*D+
\ulcorner A\urcorner)\oplus \mathcal O_{G}(bf^*D+
\ulcorner A\urcorner).$$
and 
$$\mathcal O_{F}(bf^*D+\ulcorner A\urcorner)
=\mathcal O_{F}(K_F+\ulcorner N(b, c)\urcorner).$$
$H^0$ of this sheaf does not vanish 
by the non-vanishing theorem for curves. 
Then we see $f(F)\not\subset\Bs|bD|$. 
Let $b:=q^l$ for $l\gg 0$. 
Then this is a contradiction. 
\end{proof}

\begin{cor}
Let $X$ be a projective normal surface and 
let $\Delta$ be a \Q-divisor such that 
$\llcorner \Delta\lrcorner=0$ and 
$K_X+\Delta$ is \Q-Cartier. 
If $K_X+\Delta$ is nef and big, then 
$K_X+\Delta$ is semi-ample. 
\end{cor}

\begin{proof}
Let $c$ be a positive integer 
such that $c(K_X+\Delta)$ is Cartier. 
Then we can apply Theorem~\ref{basepointfree} 
for $D:=c(K_X+\Delta)$ and $a:=2$. 
Thus $|bc(K_X+\Delta)|$ is basepoint free 
for $b\gg 0$.
\end{proof}

We would like to know 
whether the above basepoint free theorem 
holds for the case where $D\equiv 0$. 
We give the affirmative answer  
only for the case where 
$X$ has at worst rational singularities. 
But our strategy is not the X-method. 
Let us recall the following known fact. 

\begin{fact}
Let $X$ be a normal surface and 
let $\Delta$ be an effective \R-divisor. 
\begin{enumerate}
\item{If $(X, \Delta)$ is klt, then 
$X$ has at worst rational singularities.}
\item{If $X$ has at worst rational singularities, 
then $X$ is \Q-factorial. }
\end{enumerate}
\end{fact}

\begin{proof}
(1) See, for example, \cite[Theorem~14.4 and 
Remark~14.5]{Tanaka}. 

(2) See \cite[Proposition~17.1]{Lipman}.
\end{proof}

The following result is the key.

\begin{thm}\label{weakdelpezzo}
Let $X$ be a projective surface whose singularities are 
at worst rational. 
Let $\Delta$ be an \R-Weil divisor such that $\llcorner \Delta\lrcorner=0$. 
If $-(K_X+\Delta)$ is nef and big, then 
$X$ is a rational surface. 
\end{thm}

\begin{proof}
\setcounter{step}{0}
\begin{step}
In this step, we show that we may assume that 
$X$ has no curve whose self-intersection number is negative. 

Suppose the contrary, that is, 
there exists a curve $E$ in $X$ such that $E^2<0$. 
Since 
$$(K_X+E)\cdot E<(K_X+\Delta)\cdot E\leq 0,$$
we obtain a birational morphism $f:X\to Y$ 
to a projective surface whose singularities are 
at worst rational such that $\Ex f=E$. 
This follows from \cite[Theorem~6.2 and Theorem~20.4]{Tanaka}. 
Set $\Delta_Y:=f_*\Delta.$ 
We see that the discrepancy $d$, defined by 
$$K_X+\Delta=f^*(K_Y+\Delta_Y)+dE,$$
is non-negative. 
Then we can see that $-(K_Y+\Delta_Y)$ is nef and big. 
Moreover, if there exists a curve $E_Y$ in $Y$ 
such that $E_Y^2<0$, then 
we can repeat the same procedure as above. 
\end{step}

\begin{step}
In this step, we prove that we may assume that 
there exists a surjective morphism $\pi:X\to Z$ 
to a smooth projective irrational curve $Z$. 

Let $g:X'\to X$ be the minimal resolution and 
set $K_{X'}+\Delta':=g^*(K_X+\Delta)$. 
Since $-(K_X+\Delta)$ is big, 
the anti-canonical divisor 
$$-K_{X'}=-g^*(K_X+\Delta)+\Delta'$$
is also big. 
In particular, $X'$ is a ruled surface. 
If $X'$ is rational, then there is nothing to prove. 
Thus we may assume that $X'$ is an irrational ruled surface. 
Let $\theta:X'\to Z$ be its ruling. 
Because the singularities of $X$ are at worst rational, 
each curve $D$ in $\Ex (g)$ is a smooth rational curve. 
In particular, $\theta(D)$ is one point. 
This means that $\theta$ factors through $X$. 
This is what we want to show. 
\end{step}

\begin{step}
By Step~1 and \cite[Theorem~6.8]{Tanaka}, 
we see that $\rho(X)\leq 2$. 
Moreover, by Step~2, we see that $\rho(X)=2.$ 
By Step~1, we see that $-(K_X+\Delta)$ is ample 
because, for a curve $C$ in $X$, 
the equality $(K_X+\Delta)\cdot C=0$ means $C^2<0$ by Kodaira's lemma. 
Thus there are two extremal rays which induce 
the Mori fiber space to a curve by 
\cite[Theorem~6.8]{Tanaka}. 
But this contradicts $\pi:X\to Z$ and 
the irrationality of $Z$.
\end{step}
\end{proof}

In the case where $D\equiv 0$, 
the basepoint free theorem is related to 
the rationality of the log weak del Pezzo surfaces. 
Indeed, by using the above result, 
we prove the following basepoint free theorem. 

\begin{cor}[Basepoint free theorem 
in the case where $\nu=0$] 
\label{basepointfree0}
Let $X$ be a projective surface 
whose singularities are at worst rational. 
Let $\Delta$ be an \R-Weil divisor such that $\llcorner \Delta\lrcorner=0$. 
Let $D$ be a numerically trivial Cartier divisor. 
If $-(K_X+\Delta)$ is nef and big, then $D\sim 0$. 
\end{cor}

\begin{proof}
Let $f:X'\to X$ be a resolution and 
set $D':=f^*D.$ 
Since $H^0(X, D)=H^0(X', D')$, 
it is sufficient to prove $D'\sim 0$. 
By Theorem~\ref{weakdelpezzo}, 
$X'$ is rational. 
Therefore $D'\equiv 0$ means $D'\sim 0$. 
\end{proof}

\begin{rem}
In \cite{Tanaka}, 
a basepoint free theorem is established 
in the case where $X$ is a \Q-factorial surface 
(Theorem~\ref{0tanakabasepointfree}). 
But this result does not contain 
Corollary~\ref{basepointfree0}. 
On the other hand, 
a cone theorem is established 
under the assumption 
that 
$X$ is a normal surface and 
$\Delta$ is an effective \R-divisor such that 
$K_X+\Delta$ is \R-Cartier. 
For more details, see \cite{Tanaka}. 
\end{rem}

\section{Other vanishing results}

In this section, we establish some vanishing results 
other than the ones in Section~2. 
Theorem~\ref{hn-1vanish} and Theorem~\ref{fiberrelativekvv}
are the main results in this section. 
Theorem~\ref{hn-1vanish} follows from 
a fundamental inductive argument. 
Theorem~\ref{fiberrelativekvv} follows from the same argument as Section~2. 

First, we consider a generalization 
of Proposition~\ref{h2vanishing}.

\begin{prop}\label{hnvanish}
Let $X$ be an $n$-dimensional 
smooth projective variety with $n\geq 1$. 
Let $F$ be a coherent sheaf on $X$. 
Let $N$ be a nef \R-divisor with $\nu(X, N)\geq 1$. 
Then there exists a positive real number $r(F, N)$ 
such that 
$$H^n(X, F(rN+N'))=0$$
for every positive real number $r\geq r(F, N)$ and 
for every nef \R-divisor $N'$ such that 
$rN+N'$ is a \Z-divisor. 
\end{prop}

\begin{proof}
By the same argument as Theorem~\ref{h2vanishing}, 
we may assume that $F=:L$ is an invertible sheaf. 
We prove the assertion by the induction on $n=\dim X$. 
If $n=1$, then the assertion is obvious. 
Thus, we assume $n>1.$ 
Let $H$ be a smooth hyperplane section. 
Consider the exact sequence: 
\begin{eqnarray*}
0&\to& \mathcal O_X(L+rN+N')\\
&\to& \mathcal O_X(L+rN+N'+H)\\
&\to& \mathcal O_H(L+rN+N'+H)\to 0
\end{eqnarray*}
By the hypothesis of the induction, 
$H^{n-1}(H, \mathcal O_H(L+rN+N'+H))$ vanishes. 
For the vanishing of 
$H^{n}(X, \mathcal O_X(L+rN+N'+H))$, 
replacing $H$ by a large multiple, 
we can apply the Fujita vanishing theorem. 
This is what we want to show. 
\end{proof}

Second, we consider a generalization 
of the Kawamata--Viehweg type vanishing theorem 
(Theorem~\ref{kvvtype}). 

\begin{thm}\label{hn-1vanish}
Let $X$ be an $n$-dimensional 
smooth projective variety with $n\geq 2$. 
Let $B$ be a nef and big \R-divisor whose 
fractional part is simple normal crossing. 
Let $N$ be a nef \R-divisor with $\nu(X, N)\geq 1$.
Then there exists a positive real number $r(B, N)$ 
such that 
$$H^{n-1}(X, K_X+\ulcorner B\urcorner+rN+N')=0$$
for every positive integer $r\geq r(B, N)$ and 
for every nef \R-divisor $N'$ such that 
$rN+N'$ is a \Z-divisor.
\end{thm}

\begin{proof}
If $n=2$, then we obtain the assertion by Theorem~\ref{kvvtype}. 
Then, the assertion follows from the same inductive argument 
as the proof of Theorem~\ref{hnvanish}.
\end{proof}

Next, 
let us recall the following known result. 

\begin{prop}\label{fiberserrevanish}
Let $\pi:X\to S$ be a morphism 
from a proper variety $X$ to a projective variety $S$. 
Let $A_S$ be an ample Cartier divisor on $S$ and 
let $N:=\pi^*A_S$. 
Let $F$ be a coherent sheaf on $X$. 
Set $f_{\max}:=\max_{s\in S} \dim \pi^{-1}(s)$. 

If $i\geq f_{\max}+1$, then 
$$H^i(X, F(mN))=0$$
for an arbitrary integer $m\gg 0$.
\end{prop}

\begin{proof}
Consider the Leray spectral sequence 
$$E^{i,j}_2:=H^i(S, R^j\pi_*F(mN))\Rightarrow H^{i+j}(X, F(mN))=:E^{i+j}.$$
Since 
$$R^j\pi_*F(mN)=R^j\pi_*(F\otimes \pi^*(mA_S))
=R^j\pi_*(F)\otimes A_S^{\otimes m},$$
by Serre vanishing, we have $E^{i,j}_2=0$ 
for $i>0$ and $m\gg 0$. 
Thus we obtain $E^{0,j}_2=E^{j}$ for $m\gg 0$. 
If $j\geq f_{\max}+1$, then $E^{0,j}_2=0$. 
\end{proof}

By the same argument as Section~2, 
we obtain the following vanishing result. 

\begin{thm}\label{fiberrelativekvv}
Let $\pi:X\to S$ be a morphism 
from a smooth projective variety $X$ to a projective variety $S$. 
Let $A$ be a $\pi$-ample \R-divisor on $X$ whose 
fractional part is simple normal crossing. 
Let $A_S$ be an ample Cartier divisor on $S$ and 
let $N:=\pi^*A_S$. 
Set $f_{\max}:=\max_{s\in S} \dim \pi^{-1}(s)$. 

\begin{enumerate}
\item{If $i\geq f_{\max}$, then 
$$H^i(X, 
\mathcal O_X(K_X+\ulcorner A\urcorner+mN))=0$$
for an arbitrary integer $m \gg 0$.}
\item{If $i\geq f_{\max}$, then 
$$R^i\pi_*\mathcal O_X(K_X+
\ulcorner A\urcorner)=0.$$}
\end{enumerate}
\end{thm}

\begin{proof}
By the usual spectral sequence argument, 
$(2)$ follows from $(1)$. 
Thus we only prove $(1)$. 
By the assumption, we may assume that 
$A$ is an ample \R-divisor whose fractional part is simple normal crossing. 
Moreover, by Proposition~\ref{cycliccover}, 
we may assume that $A$ is an ample \Z-divisor. 
Then, the assertion follows from 
the same arguments as Proposition~\ref{h1proto} 
by using Proposition~\ref{fiberserrevanish} 
instead of Proposition~\ref{h2vanishing}.  
\end{proof}

\section{Examples in dimension three}

It is natural to consider the following question. 

\begin{ques}
Can we generalize the vanishing results 
in Section~2 to higher dimensional varieties?
\end{ques}

Unfortunately, the answer is NO. 
In this section, we construct 
counter-examples.

\begin{ex}[cf. Proposition~\ref{h2vanishing}]
There exists a smooth projective 3-fold $X$, 
a coherent sheaf $F$ and 
a semi-ample and big \Z-divisor $B$ 
which satisfy the following property. 

There exists a positive integer $m_0$ 
such that for an arbitrary integer $m\geq m_0$ 
$$H^2(X, F(mB))\neq 0.$$
\end{ex}

\begin{proof}[Construction]
Let $X_0$ be an arbitrary smooth projective $3$-fold and 
let $x_0\in X_0$ be an arbitrary point. 
Let $f:X\to X_0$ be the blowup at $x_0$. 
Let $E$ be the exceptional divisor and 
let $B:=f^*A_0$ where $A_0$ is an ample \Z-divisor on $X_0$. 
We define $F$ by 
$$F:=\mathcal O_X(K_X+E).$$
Consider the exact sequence: 
$$0\to \mathcal O_X(K_X+mB)
\to \mathcal O_X(K_X+E+mB)
\to \mathcal O_E(K_E+mB)\to 0.$$
$H^3(X, \mathcal O_X(K_X+mB))$ 
vanishes 
for an arbitrary large integer $m\gg 0$ 
by Proposition~\ref{hnvanish}. 
Consider $H^2(E, \mathcal O_E(K_E+mB))$. 
Since $B=f^*(A)$ and $f(E)$ is one point, 
we have 
$$h^2(E, K_E+mB)=h^2(E, K_E)=h^0(E, \mathcal O_E)=1\neq 0$$
for an arbitrary integer $m\in \mathbb{Z}$. 
These mean 
$$H^2(X, \mathcal O_X(K_X+E+mB))\neq 0$$
for an arbitrary large integer $m\gg 0$. 
This is what we want to show. 
\end{proof}

In the construction of the following three examples, 
we use a counter-example to the Kodaira vanishing theorem 
(cf. \cite{Raynaud}).

\begin{ex}[cf. Theorem~\ref{h1proto}]
\label{cexnu=1}
There exists a smooth projective 3-fold $X$, 
an ample \Z-divisor $A$ and 
a semi-ample \Z-divisor $N$ with $\nu(X, N)=1$ 
which satisfy the following property. 

There exists a positive integer $m_0$ 
such that for an arbitrary integer $m\geq m_0$ 
$$H^1(X, K_X+A+mN)\neq 0.$$
\end{ex}

\begin{proof}[Construction]
Let $Z$ be a smooth projective surface and 
let $A_Z$ be an ample \Z-divisor such that 
$$H^1(Z, K_{Z}+A_Z)\neq 0.$$
Let $C$ be an arbitrary smooth projective curve. 
Set $X:=Z\times C$ and 
let $\pi_Z$ and $\pi_C$ be their projections respectively. 
Take two distinct points $c_0\in C$ and $c_1\in C$ and 
let $Z_0:=Z\times \{c_0\}$ and $Z_1:=Z\times \{c_1\}$. 
Note that $Z_0\simeq Z$ and $Z_1\simeq Z$. 
Since $c_0$ and $c_1$ are ample \Z-divisors on $C$, 
$$A:=\pi_Z^*(A_Z)+Z_0+Z_1\,\,\,{\rm and}\,\,\,A-Z_0$$
are ample \Z-divisors on $X$. 
We show 
$$H^1(X, K_X+A+mZ_1)\neq 0$$
for an arbitrary integer $m\gg 0$. 
Consider the exact sequence: 
\begin{eqnarray*}
0&\to& \mathcal O_X(K_X+A-Z_0+mZ_1)\\
&\to& \mathcal O_X(K_X+A+mZ_1)\\
&\to& \mathcal O_{Z_0}(K_{X}+A+mZ_1)\to 0.
\end{eqnarray*}
By Theorem~\ref{hn-1vanish}, 
$H^2(X, \mathcal O_X(K_X+A-Z_0+mZ_1))$ 
vanishes for an arbitrary large integer $m\gg 0.$ 
Let us calculate $H^1(Z_0, \mathcal O_{Z_0}(K_{X}+A+mZ_1))$. 
By $Z_0\cap Z_1=\emptyset$, we see 
$$H^1(Z_0, K_{X}+A+mZ_1)
=H^1(Z_0, K_{Z_0}+\pi_Z^*A_Z)
=H^1(Z, K_{Z}+A_Z)\neq 0$$
for an arbitrary integer $m\in \mathbb{Z}.$ 
These mean 
$$H^1(X, K_X+A+mZ_1)\neq 0$$
for $m\gg 0.$ 
This is what we want to show. 
\end{proof}

\begin{ex}[cf. Theorem~\ref{h1proto}]
\label{cexnu=3}
There exists a smooth projective 3-fold $X$, 
an ample \Z-divisor $A$ and 
a  semi-ample and big \Z-divisor $B$ 
which satisfy the following property. 

There exists a positive integer $m_0$ 
such that for an arbitrary integer $m\geq m_0$ 
$$H^1(X, K_X+A+mB)\neq 0.$$
\end{ex}

\begin{proof}[Construction]
By Proposition~\ref{cycliccover} and 
Step~1 in the proof of Theorem~\ref{kvvtype}, 
it is sufficient that we construct 
$$H^1(X, K_X+\Delta+A+mN)\neq 0$$ 
for an ample \Q-divisor $A$ and 
a simple normal crossing \Q-divisor $\Delta$ such that 
$0\leq \Delta\leq 1$ and that $\Delta+A$ is a \Z-divisor. 

Let $Z\subset \mathbb{P}^N$ be a smooth projective surface and 
let $A_Z$ be an ample \Z-divisor such that 
$$H^1(Z, K_{Z}+A_Z)\neq 0.$$ 
Let $Y\subset \mathbb{P}^{N+1}$ be the projective cone over $Z$ and 
let $f:X\to Y$ be the blowup of the vertex of $Y$. 

Then, by \cite[Chapter~V, Example~2.11.4]{Hartshorne}, 
we see that 
$X=\Proj_Z(\mathcal O_Z\oplus \mathcal O_Z(1)).$ 
Let $\pi:X\to Z$ be the natural projection and 
let $\mathcal O_X(1)$ be the canonically defined 
$\pi$-ample invertible sheaf. 
Let $Z_0$ and $Z_1$ be the sections 
defined by the following surjections respectively 
$$\mathcal O_Z\oplus \mathcal O_Z(1)\to \mathcal O_Z\to 0$$ 
$$\mathcal O_Z\oplus \mathcal O_Z(1)\to \mathcal O_Z(1)\to 0.$$
By the definition of $Z_0$ and $Z_1$, 
we have $\mathcal O_X(1)|_{Z_0}=\mathcal O_{Z_0}$ and 
$\mathcal O_X(1)|_{Z_1}=\mathcal O_{Z_1}(1)$ where 
$\mathcal O_{Z_1}(1)$ 
is a very ample invertible sheaf defined by 
$Z\simeq Z_1$ and $\mathcal O_{Z}(1)$. 
By the same argument as \cite[Chapter~V, Proposition~2.6]{Hartshorne}, 
we see $\mathcal O_X(Z_1)\simeq \mathcal O_X(1)$. 
Moreover, by a direct calculation, 
we see that $Z_0$ is the exceptional locus of $f$.
(See \cite[Chapter~V, Example~2.11.4]{Hartshorne}.) 
Note that $Z_0$ and $Z_1$ are disjoint because $\mathcal O_X(1)|_{Z_0}=\mathcal O_{Z_0}$. 

We fix a small positive rational number 
$\epsilon_0\in \mathbb{Q}_{>0}$
such that the \Q-divisor 
$$A:=\pi^*A_Z+\epsilon_0 Z_1$$
is ample. 
Set 
$$\Delta:=(1-\epsilon_0) Z_1+Z_0.$$
Note that $\Delta+A$ is a \Z-divisor. 
Let $A_Y$ be an ample invertible sheaf on $Y$ 
and let $$B:=f^*A_Y.$$
Consider the exact sequence: 
\begin{eqnarray*}
0&\to& \mathcal O_X(K_X+\Delta+A+mB-Z_0)\\
&\to& \mathcal O_X(K_X+\Delta+A+mB)\\
&\to& \mathcal O_{Z_0}(K_X+\Delta+A+mB)\to 0.
\end{eqnarray*}
First let us calculate $H^2(X, 
\mathcal O_X(K_X+\Delta+A+mB-Z_0))$: 
\begin{eqnarray*}
&&H^2(X, K_X+\Delta+A+mB-Z_0)\\
&=&H^2(X, K_X+(1-\epsilon_0)Z_1+A+mB)=0
\end{eqnarray*}
for an arbitrary large integer $m\gg 0$ 
by Theorem~\ref{hn-1vanish}. 
Second let us calculate 
$H^1(Z_0, \mathcal O_{Z_0}(K_X+\Delta+A+mB))$. 
Since $f(Z_0)$ is one point, 
we see 
\begin{eqnarray*}
&&\mathcal O_{Z_0}(K_X+\Delta+A+mB)\\
&=&\mathcal O_{Z_0}(K_X+Z_1+Z_0+\pi^*A_Z+mf^*A_Y)\\
&=&\mathcal O_{Z_0}(K_{Z_0}+\pi^*A_Z)
\end{eqnarray*}
for an arbitrary integer $m\in \mathbb{Z}.$ 
By $Z_0 \simeq Z$, we see 
\begin{eqnarray*}
H^1(Z_0, K_X+\Delta+A+mB)&=&
H^1(Z_0, K_{Z_0}+\pi^*A_Z)\\
&=&H^1(Z, K_{Z}+A_Z)\\
&\neq&0.
\end{eqnarray*}
Therefore we obtain 
$$H^1(X, K_X+\Delta+A+mB)\neq 0$$
for an arbitrary large integer $m\gg 0$. 
This is what we want to show. 
\end{proof}

\begin{ex}[cf. Theorem~\ref{h1proto}]
\label{cexnu=2}
There exists a smooth projective 3-fold $W$, 
an ample \Z-divisor $A_W$ and 
a semi-ample \Z-divisor $N_W$ with $\nu(W, N_W)=2$ 
which satisfy the following property. 

There exists a positive integer $m_0$ 
such that for an arbitrary integer $m\geq m_0$ 
$$H^1(W, K_W+A_W+mN_W)\neq 0.$$
\end{ex}

\begin{proof}[Construction]
By Proposition~\ref{cycliccover} and 
Step~1 in the proof of Theorem~\ref{kvvtype}, 
it is sufficient that we construct 
$$H^1(W, K_W+\Delta_W+A_W+mN_W)\neq 0$$ 
for an ample \Q-divisor $A_W$ and 
a simple normal crossing \Q-divisor $\Delta_W$ such that 
$0\leq \Delta_W\leq 1$ and that $\Delta_W+A_W$ is a \Z-divisor. 

We use the same notations 
$X, Y, Z, A,\cdots$ as Example~\ref{cexnu=3}. 
Let $y_0\in Y$ be the vertex as a projective cone. 
There exists a finite morphism 
$\theta:Y\to \mathbb{P}^3.$ 
Fix an open dense subset $\mathbb{A}^3\subset \mathbb{P}^3$ 
such that $\theta(y_0)\in \mathbb{A}^3.$ 
Take an arbitrary projection 
$\mathbb{A}^3\to \mathbb{A}^2=:U$ and 
fix its projectivication 
$U\subset \mathbb{P}^2=:P$. 
Now, we have the following morphisms 
\begin{eqnarray*}
X\overset{f}\to Y\overset{\theta}\to\mathbb{P}^3\supset \mathbb{A}^3 
\to\mathbb{A}^2=U\subset \mathbb{P}^2=P.
\end{eqnarray*}
Here, by considering 
the composition of the above dominant rational maps, 
we obtain a dominant rational map 
$g:X\dashrightarrow P$. 
Note that this is a morphism on 
$(\theta\circ f)^{-1}(\mathbb{A}^3)$. 
By its construction, 
we see $Z_0\subset 
(\theta\circ f)^{-1}(\mathbb{A}^3)$ 
since $f(Z_0)=y_0.$ 
By taking a log resolution 
of indeterminacy $h:W\to X$, 
we obtain a surjective morphism $l:W\to P$ 
from a smooth projective $3$-fold $W$ such that 
$h(\Ex (h))\subset X\setminus 
(\theta\circ f)^{-1}(\mathbb{A}^3)$ and 
$h^{-1}(Z_1)\cup\Ex (h)$ is simple normal crossing 
(cf. \cite{CP}). 
Then $h:h^{-1}(Z_0)\to Z_0$ is an isomorphism and 
let $Z_W:=h^{-1}(Z_0).$ 
Let $$A_W:=h^*A-E=h^*\pi^*A_Z+\epsilon_0 h^*Z_1-E$$ be an ample \Q-divisor on $W$ where 
$E$ is an $h$-exceptional \Q-divisor with $0\leq E<1$. 
Note that we can find such a divisor $E$ by \cite[Lemma~2.62]{KM}. 
Let 
$$\Delta_W:=Z_W+\ulcorner A_W\urcorner-A_W.$$
Since $\Supp (h^*Z_1\cup \Ex(h))$ 
is disjoint from $Z_W$ we see that 
$$\mathcal O_{Z_W}(\Delta_W+A_W)=\mathcal O_{Z_W}(Z_W+h^*\pi^*A_Z).$$
Let $A_P$ be an ample \Z-divisor on $P$ and 
let
$$N_W:=l^*A_P.$$
Consider the exact sequence: 
\begin{eqnarray*}
0&\to& \mathcal O_W(K_W+\Delta_W+A_W+mN_W-Z_W)\\
&\to& \mathcal O_W(K_W+\Delta_W+A_W+mN_W)\\
&\to& \mathcal O_{Z_W}(K_W+\Delta_W+A_W+mN_W)\to 0.
\end{eqnarray*}
Then by the same calculation as Example~\ref{cexnu=3}, 
we obtain the desired result. 
\end{proof}

\section*{Acknowledgement}
The author would like to 
thank Professor Osamu Fujino 
for many comments and discussions. 
He thanks Professor Atsushi Moriwaki 
for warm encouragement.

\end{document}